\documentclass[12pt,a4paper]{amsart}
\usepackage{graphicx}
\usepackage{color}

\usepackage{amsmath, amssymb, amsthm, amsfonts}

\allowdisplaybreaks
\usepackage{amsmath,amssymb,amsthm,amsfonts,mathrsfs,enumerate,sansmath,mathtools,amscd
}
\usepackage[colorlinks=true]{hyperref}
\hypersetup{urlcolor=blue, citecolor=blue, draft=false}

\newtheorem{theorem}{Theorem}[section]
\newtheorem{proposition}[theorem]{Proposition}

\newtheorem{corollary}[theorem]{Corollary}
\theoremstyle{definition}
\newtheorem{definition}[theorem]{Definition}

\newtheorem{remark}[theorem]{Remark}

\begin{document}
\title[Carath\'eodory density of the Hurwitz metric]{Carath\'eodory density of the Hurwitz metric on plane domains}

\author[Arstu and S. K. Sahoo]{Arstu$^\dagger$ and Swadesh Kumar Sahoo$^\dagger$}

\address{$^\dagger$Discipline of Mathematics, Indian Institute of Technology Indore, Simrol, Khandwa Road, Indore 453 552, India}
\email{arstumothsra@gmail.com}
\email{swadesh.sahoo@iiti.ac.in}

\subjclass[2010]{Primary 30F45; Secondary 30C20, 30C80}

\keywords{hyperbolic density, Hurwitz density, Kobayashi density, Carath\'eodory density, conformal mapping, covering mapping}

\begin{abstract}
It is well-known that the Carath\'eodory metric is a natural generalization of the Poincar\'e metric, namely, the hyperbolic metric of the unit disk.
In 2016, the Hurwitz metric was introduced by D. Minda in arbitrary proper subdomains of the complex plane and he proved that this metric coincides with the hyperbolic metric when the domains are simply connected.
In this paper, we define a new metric which generalizes the Hurwitz metric in the sense of Carath\'eodory. Our main focus is to study its various basic properties
in connection with the Hurwitz metric.
\end{abstract}
\maketitle

	
\section{Introduction}\label{sec1}
Studying families of holomorphic functions associated with the hyperbolic metric always remains the hot topic in geometric function theory. Several researchers (for instance see \cite{Dav08,Min83,Min16}) have introduced new metrics which are closely related to the hyperbolic metric and established their comparison properties in possible situations. In particular,
the generalized Kobayashi metric is one such metric which is always greater than or equal to the hyperbolic metric (see \cite[Proposition~1]{Kee07}).
The generalized Kobayashi metric in a domain is defined by the smallest push forward of the hyperbolic metric 
from a hyperbolic domain to a plane domain by holomorphic functions. In \cite{Kee07}, it is further proved that the generalized Kobayashi metric agrees with the hyperbolic metric on simply connected domains (see also \cite{ke07}). 
Coincidence of the hyperbolic and the generalized Kobayashi densities on other plane domains are studied in
\cite{Tav09}.
Moreover,
the Kobayashi density satisfies the generalized Schwarz lemma for holomorphic function between two domains. 

In 2016, Minda introduced a new metric, namely, the Hurwitz metric that also exceeds the hyperbolic metric in hyperbolic domains (see \cite{Min16}) and investigated several basic properties such as distance decreasing property,
conformal invariance property, domain monotonicity property, bilipschitz equivalent properties with the hyperbolic and the quasihyperbolic metrics.
In our recent work (see \cite{AS1}), we studied a new
metric that generalizes the Hurwitz metric in the sense of Kobayashi. This new work focuses on some basic properties of this generalized metric.

On the other hand, the classical Carath\'eodory metric is another generalized metric which is always less than or equal to the hyperbolic metric. The Carath\'eodory metric in a domain is the largest pull back of the hyperbolic metric. Similar to the case of the Kobayashi metric, the Carath\'eodory metric also agrees with the hyperbolic metric on simply connected domains. Furthermore, it satisfies the generlized Schwarz lemma for holomorphic function between two domains. Analogous to the Carath\'odory metric, in this paper, we generalize the Hurwitz metric and study its basic properties.

Rest of this document is organized as follows: Section \ref{p2sec2} contains preliminary information including terminology, definitions and well known results. We define the generalized Hurwitz metric in the sense of Carath\'eodory Section~\ref{p2sec3} and study its various properties including distance decreasing property for special class of holomorphic function between two domains. Finally, Section~\ref{p2sec4} is devoted to the distance between two points induced by the generalized Hurwitz metric.

\section{preliminaries}\label{p2sec2}

Throughout the paper, unless it is specified, we assume that $\Omega$
is an arbitrary domain and $Y$ is a proper subdomain in $\mathbb{C}$,
the complex plane. Symbolically, we write $\Omega\subset \mathbb{C}$ and $Y\subsetneq \mathbb{C}$.
We denote $\mathcal{H}(\Omega,Y)$ by the set of all holomorphic functions from $\Omega$ into $Y$.
For a fixed $w\in\Omega$, we define the following notation:
$$
\mathcal{H}^s_w(\Omega,Y)
=\{h\in\mathcal{H}(\Omega,Y),~ h(w)=s,~h(z)\neq s~\mbox{ for all } z\in\Omega\setminus\{w\}\}.
$$
The open unit disk $\{z\in\mathbb{C}:\,|z|<1\}$ is denoted by $\mathbb{D}$. The family $\mathcal{H}^0_w(\mathbb{D},Y)$ is known as the {\em Hurwitz family}. More about the Hurwitz family and several other classes of holomorphic functions analogous to the Hurwitz family are discussed in \cite{Min16}. By setting 
$$
F'(0)=r_{Y}(w)=\max\{h'(0):h\in\mathcal{H}^0_w(\mathbb{D},Y)\},
$$ 
the {\em Hurwitz density} is defined as
$$
\eta_Y(w)=\cfrac{2}{F'(0)}=\cfrac{2}{r_Y(w)}.
$$
An equivalent definition of the Hurwitz density can be found in \cite{AS1}. A domain $Y$ is said to be hyperbolic provided its complement $\mathbb{C}\setminus Y$ contains at least two points. 
The supremum of $\{|h'(0)|:h\in\mathcal{H}(\mathbb{D},Y)\}$ leads to the definition of the hyperbolic density. Indeed, for any point $w\in Y$ the {\em hyperbolic density} $\lambda_Y$ is defined as 
$$
\lambda_Y(w)=\cfrac{2}{|g'(0)|},
$$ 
where $g$ is a universal covering mapping from $D$ onto $Y$. Note that existence of such $g$ is guaranteed by the uniformization theorem.
Analogous to the hyperbolic density, we now describe the maximizer for the Hurwitz family. For $s\in Y$, there exists a holomorphic covering map $F:\mathbb{D}\setminus\{0\}\to Y\setminus\{s\}$ 
which extends to $\mathbb{D}$ holomorphically in such a way that $F(0)=s$ and $F'(0)>0$.
This is determined by the subgroup generated by the curve, namely the circle centered at $s$ and radius $\rho$, for some $\rho>0$, in $Y\setminus\{s\}$ of the fundamental group of $Y\setminus\{s\}$. The function $F$ is the unique extremal function for the Hurwitz-extremal problem $\max\{h'(0):h\in\mathcal{H}^w_0(\mathbb{D},Y),~h'(0)>0\}$ and is defined as the {\em Hurwitz covering} map (see \cite{Min16}). 


\section{Carath\'eodory density of the Hurwitz metric}\label{p2sec3}
In \cite{AS1}, by adopting the idea of the Kobayashi metric, we generalized the Hurwitz density for a domain $\Omega\subset\mathbb{C}$
and $Y\subsetneq \mathbb{C}$ as follows:
$$
\eta_\Omega^Y(w)=\inf\frac{\eta_Y(s)}{|h'(s)|},
$$
where $\eta_Y$ is the Hurwitz density on $Y$ and the infimum is taken over all $h\in\mathcal{H}(Y,\Omega)$ satisfying $h(s)=w$, $h(t)\neq w$ for all $t\in Y\setminus\{s\}$, and $h'(s)\neq 0$.
We name the quantity $\eta_\Omega^Y$ by {\it the Kobayashi 
density of the Hurwitz metric of $\Omega$ relative to $Y$}.

As stated at the end of Section~\ref{sec1}, this section is devoted to the introduction of a new density that generalizes the Hurwitz density in the sense of Carath\'eodory. This is defined as follows:

\begin{definition}\label{HC}
Let $w\in\Omega\subsetneq\mathbb{C}$. 
For an element $s\in\mathbb{D}$, we define a new quantity
\begin{equation}\label{p2eq1}
\mathscr{C}_{\Omega}^{\mathbb{D},s}(w)=\sup\eta_{\mathbb{D}}(h(w))|h'(w)|,
\end{equation}
where the supremum is taken over all $h\in \mathcal{H}(\Omega,\mathbb{D})$ such that 
$h(w)=s$, $h(z)\neq s$ for all $z\in\Omega\setminus\{w\}$, i.e. 
for all $h\in \mathcal{H}^s_w(\Omega,\mathbb{D})$.
We call this quantity by {\em the Carath\'eodory
density of the Hurwitz metric of $\Omega$ relative to $\mathbb{D}$. Setting $\mathscr{C}_{\Omega}^{\mathbb{D}}:=\mathscr{C}_{\Omega}^{\mathbb{D},0}.$}
\end{definition}


\begin{remark}
\begin{enumerate}
\item 
Note that on simply connected domains the Hurwitz density agrees with the hyperbolic density, so one can replace $\eta_{\mathbb{D}}$ by the hyperbolic density $\lambda_{\mathbb{D}}$ in Definition~\ref{p2eq1}. 
\item
If $\Omega=\mathbb{C}$, then by Liouville's theorem, the only holomorphic function from $\Omega$ into $\mathbb{D}$ is a constant function, which does not belong to the class
$\mathcal{H}^s_w(\Omega,\mathbb{D})$. Hence, it can be defined that
$\mathscr{C}_{\mathbb{C}}^{\mathbb{D},s}(w)=0$ when the set $\mathcal{H}^s_w(\Omega,\mathbb{D})$ becomes empty. It suggests us to assume that $\mathcal{H}^s_w(\Omega,Y)\neq \emptyset$ throughout the paper for an arbitrary base domain $Y\subsetneq\mathbb{C}$.
\end{enumerate}
\end{remark}


The first basic property of the Carath\'eodory
density of the Hurwitz metric $\mathscr{C}_{\Omega}^{\mathbb{D},s}$ is that the supremum is attained by some holomorphic function $h\in\mathcal{H}^s_w(\Omega,\mathbb{D})$ in  \eqref{p2eq1}. 

\begin{proposition}\label{p2prop1}
Let $\Omega\subsetneq\mathbb{C}$ be a domain
and $\mathcal{H}_w^0(\Omega,\mathbb{D})\neq \emptyset$. Then, the Carath\'eodory
density of the Hurwitz metric $\mathscr{C}_{\Omega}^{\mathbb{D}}$ can be computed by the formula: 
$$\mathscr{C}_{\Omega}^{\mathbb{D}}(w)=2\,\max\{|h'(w)|:\, h\in\mathcal{H}^0_w(\Omega,\mathbb{D})\}.
$$
\end{proposition}
\begin{proof}
Since the members of the family $\mathcal{H}^0_w(\Omega,\mathbb{D})$ are uniformly bounded by $1$, by Montel's Theorem, $\mathcal{H}^0_w(\Omega,\mathbb{D})$ is a normal family. 
By Definition~\ref{HC} there exists a sequence of holomorphic functions $h_n\in\mathcal{H}^0_w(\Omega,\mathbb{D})$ such that $2|h_n'(w)|\to\mathscr{C}_{\Omega}^{\mathbb{D}}(w),$
since $h_n(w)=0$ and $\eta_\mathbb{D}(0)=\lambda_\mathbb{D}(0)=2$. Furthermore, by the open mapping theorem, there exists a subsequence $h_{n_k}$ of $h_n$ which converges to either an open map $h$ or a constant map. Since $h_n\in\mathcal{H}(\Omega,\mathbb{D})$, it follows that $|h(z)|\leq 1$ for all $z\in\Omega$. Note that, if $h(z)$ attains $1$ for some $z\in\Omega$, then by the maximum modulus principle, $|h|= 1,$ contradicting to the fact that $h(w)=0.$ Moreover, by Hurwitz Theorem, there exists an $N\in\mathbb{N}$ such that $h_{n_k}$ and $h$ have 
the same number of zeros for all $n_k\ge N$ in some neighborhood of $w$. Since $h(z)\neq 0$ for all $z\in\Omega\setminus\{w\}$, we conclude by the uniqueness of limit that 
$2|h'(w)|=\mathscr{C}_{\Omega}^{\mathbb{D}}(w),$ which completes the proof.
\end{proof}

\begin{remark}
By a suitable composition of the disk automorphism with the function obtained in Proposition~\ref{p2prop1}, we can prove the existence of the holomorphic function $h$ in
Definition~\ref{HC} when $s\neq 0$. 
\end{remark}


Alike to the case of coinciding of the hyperbolic and Carath\'eodory density on simply connected domains, we now prove that the Hurwitz density
$\eta_{\Omega}$ and the Carath\'eodory density of the Hurwitz metric $\mathscr{C}_{\Omega}^{\mathbb{D},s}$ too agree on simply connected domains $\Omega$.

\begin{proposition}\label{p2prop2}
If $\Omega\subsetneq\mathbb{C}$ is a simply connected domain, then
the Carath\'eodory density of the Hurwitz metric $\mathscr{C}_{\Omega}^{\mathbb{D},s}$ coincides with the Hurwitz density $\eta_\Omega$ as well as with the Kobayashi density of the Hurwitz metric $\eta_{\Omega}^{\mathbb{D}}$. That is, we have
$\mathscr{C}_{\Omega}^{\mathbb{D},s}\equiv\eta_{\Omega}\equiv \eta_{\Omega}^{\mathbb{D}}.$
\end{proposition}
\begin{proof}
By the distance decreasing property of the Hurwitz density (see \cite[Theorem~6.1]{Min16}), for a point $w\in\Omega$ and for any $h\in\mathcal{H}^s_w(\Omega,\mathbb{D})$ we have $\eta_{\mathbb{D}}(h(w))|h'(w)|\leq\eta_{\Omega}(w).$ 
By taking supremum over all $h\in \mathcal{H}^s_w(\Omega,\mathbb{D})$, in one hand, we obtain $\mathscr{C}_{\Omega}^{\mathbb{D},s}(w)\leq\eta_{\Omega}(w).$ On the other hand, to prove the reverse inequality, we consider the conformal homeomorphism $f:\Omega\to \mathbb{D}$ which is guaranteed by Riemann Mapping Theorem. By \cite[Corollary~6.2]{Min16}, it follows that
$$
\eta_\Omega(w)=\eta_\mathbb{D}(h(w))|h'(w)|
\le \mathscr{C}_{\Omega}^{\mathbb{D},s}(w),
$$
where the inequality holds by Definition~\ref{HC}. Thus, we have the identity 
$\mathscr{C}_{\Omega}^{\mathbb{D},s}\equiv\eta_{\Omega}$. 

The second required identity follows from \cite[Corollary~3.10]{AS1}, completing the proof.
\end{proof}


Due to \cite[Corollary~3.10]{AS1}, the Kobayashi density of the Hurwitz metric $\eta_{\Omega}^{\mathbb{D}}$ and the Hurwitz density $\eta_{\Omega}$ both agree on any domain $\Omega$, whereas, in the following result we show that on non-simply connected domains the Carath\'eodory density of the Hurwitz metric $\mathscr{C}_{\Omega}^{\mathbb{D},s}$ is strictly less than the Hurwitz density $\eta_{\Omega}$.

\begin{proposition}\label{p2prop3} Let
$\Omega\subsetneq\mathbb{C}$ be a non-simply connected domain and $\mathscr{C}_{\Omega}^{\mathbb{D}}>0$. Then for an element $w\in\Omega$ we have the strict inequality:
$
\mathscr{C}_{\Omega}^{\mathbb{D}}(w)<\eta_{\Omega}(w).
$ 
\end{proposition}

\begin{proof}
Let $w\in\Omega.$ Since $\Omega\subsetneq\mathbb{C},$ there exists a Hurwitz covering map $g:\mathbb{D}\to \Omega$ with $g(0)=w.$ By Proposition \ref{p2prop1}, there exists a function $h\in\mathcal{H}^0_w(\Omega,\mathbb{D})$
such that 
\begin{equation}\label{p2eq}
\mathscr{C}_{\Omega}^{\mathbb{D}}(w)=2|h'(w)|=\eta_{\mathbb{D}}(h(w))|h'(w)|
\end{equation}
holds, since $h(w)=0$ and $\eta_\mathbb{D}(0)=\lambda_\mathbb{D}(0)=2$.
Thus, we observe that the composition $h\circ g$ is a holomorphic function from $\mathbb{D}$ to $\mathbb{D}$ that fixes the origin. Since $\Omega$ is non-simply connected, 
the covering map $g$ can not be one-one and hence the composition $h\circ g$ can never be conformal. Thus, by the classical Schwarz lemma we conclude the strict inequality
$$
\lambda_{\mathbb{D}}((h\circ g)(0))|(h\circ g)'(0)|<\lambda_{\mathbb{D}}(0).
$$ 
Note that the hyperbolic density coincides with the Hurwitz density on simply connected hyperbolic domains (see \cite[p.~15]{Min16}). Therefore, it follows that
\begin{equation}\label{p2eq2}
\eta_{\mathbb{D}}((h\circ g)(0))|(h\circ g)'(0)|<\eta_{\mathbb{D}}(0).
\end{equation}
Since $g$ is a Hurwitz covering map, by 
\cite[Theorem~6.1]{Min16}, we have the equality
\begin{equation}\label{p2eq3}
\eta_{\Omega}(g(0))|g'(0)|=\eta_{\mathbb{D}}(0).
\end{equation}
Combining \eqref{p2eq2} and \eqref{p2eq3}, we obtain from \eqref{p2eq} that
$$
\mathscr{C}_{\Omega}^{\mathbb{D}}(w)=\eta_{\mathbb{D}}(h(w))|h'(w)|
=\eta_{\mathbb{D}}((h\circ g)(0))
\cfrac{|(h\circ g)'(0)|}{|g'(0)|}<\cfrac{\eta_{\mathbb{D}}(0)}{|g'(0)|}=\eta_{\Omega}(w),
$$ 
where the second equality follows by the chain rule.
\end{proof}


Since the Hurwitz density can be defined on a proper subdomain of the complex plane, a natural way of further generalizing the Carath\'eodory density of the Hurwitz metric $\mathscr{C}_\Omega^{\mathbb{D},s}$ by changing the base domain from the unit disk to a proper subdomain $Y$ of $\mathbb{}C$. The definition is as follows: 

\begin{definition}\label{p2def3.7}
Let $Y\subsetneq\mathbb{C}$ and $\Omega\subset \mathbb{C}$ be domains. For $w\in\Omega$ and $s\in Y$, the {\em Carath\'eodory density of the Hurwitz metric of $\Omega$ relative to the base domain $Y$} is defined as 
$$
\mathscr{C}_{\Omega}^{Y,s}(w)=\sup\eta_Y(h(w))|h'(w)|,
$$
where the supremum is taken over all $h\in \mathcal{H}(\Omega,Y)$ such that 
$h(w)=s$, $h(z)\neq s$ for all $z\in\Omega\setminus\{w\}$, i.e. 
for all $h\in \mathcal{H}^s_w(\Omega,Y)$.
\end{definition}


In \cite{AS1} we have noticed that the Kobayashi density of the Hurwitz metric $\eta_\Omega^Y$ exceeds over the Hurwitz density $\eta_\Omega$ whereas in case of the Carath\'eodory density of the Hurwitz metric $\mathscr{C}_{\Omega}^{Y,s}$, we prove that it lacks the Hurwitz density on proper subdomains of $\mathbb{C}$.

\begin{proposition}\label{p2prop4}
Let $\Omega$ and $Y$ be proper subdomains of the complex plane $\mathbb{C}$. If for an element $s\in Y$, we assume $\mathscr{C}_{\Omega}^{Y,s}>0$ then 
$$
\eta_{\Omega}(w)\geq\mathscr{C}_{\Omega}^{Y,s}(w)
$$ 
holds for every $w\in\Omega$.
\end{proposition}
\begin{proof}
By the distance decreasing property of the Hurwitz density, for $w\in\Omega,s\in Y$ and for any $h\in\mathcal{H}^s_w(\Omega,Y)$ we have
$$
\eta_Y(h(w))|h'(w)|\leq\eta_{\Omega}(w).
$$
Taking the supremum over all $h\in\mathcal{H}^s_w(\Omega,Y)$ on both sides, we obtain
$$
\mathscr{C}_{\Omega}^{Y,s}(w)\leq\eta_{\Omega}(w).
$$ 
Since $w\in\Omega$ was arbitrary, we conclude the proof as desired.
\end{proof}


Recall that the Hurwitz density and the hyperbolic density agree on simply connected domains. Analogous to this, we now prove that upon some specific conditions the Carath\'eodory density of the Hurwitz metric $\mathscr{C}_{\Omega}^{Y,s}$ and the Hurwitz density $\eta_{\Omega}$ coincide and in a more special situation, they also coincide with the Kobayashi density of the Hurwitz metric $\eta_\Omega^Y$.

\begin{proposition}\label{p2prop5}
Let $\Omega,Y\subsetneq \mathbb{C}$ be domains. 
Suppose that for every $s\in Y$ there exists a point $w\in\Omega$ and a holomorphic covering map 
$g_s:\Omega\setminus\{w\}\to Y\setminus\{s\}$ which extends to a holomorphic function 
$g:\Omega\to Y$ with $g(w)=s$ and $g'(w)\neq 0$. If $\mathscr{C}_{\Omega}^{Y,s}>0$, then 
$$
\mathscr{C}_{\Omega}^{Y,s}\equiv\eta_{\Omega}.
$$
In particular, when $Y=\Omega$, we have
$$
\mathscr{C}_{\Omega}^{\Omega,w}\equiv\eta_{\Omega}\equiv\zeta_{\Omega}^{\Omega}.
$$
\end{proposition}
\begin{proof}
By the distance decreasing property of the Hurwitz density (see the second part of \cite[Theorem~6.1]{Min16}), we have
$$
\eta_Y(g(w))|g'(w)|=\eta_{\Omega}(w).
$$
Now, plugging the holomorphic covering map $g_s$ into Definition~\ref{p2def3.7}, in one hand we obtain 
$$
\mathscr{C}_{\Omega}^{Y,s}(w)\ge \eta_Y(g(w))|g'(w)|=\eta_{\Omega}(w).
$$ 
On the other hand, the reverse inequality follows from Proposition \ref{p2prop4}. Since $w$ is arbitrary, the Carath\'eodory density of the Hurwitz metric $\mathscr{C}_{\Omega}^{Y,s}$ and the Hurwitz density
$\eta_{\Omega}$ both agree over $\Omega$.

The proof of the second part is a combination of the above identity that we just proved and the identity proved in \cite[Proposition~3.9]{AS1}.
\end{proof}

An instant corollary to Proposition~\ref{p2prop5} is that on simply connected domains both the Hurwitz density $\eta_{\Omega}$ and the Carath'eodory density of the Hurwitz metric $\mathscr{C}_{\Omega}^{Y,s}$ agree.

\begin{corollary}
If $\Omega\subsetneq\mathbb{C}$ be a simply connected domain and $Y\subsetneq\mathbb{C}$ be an arbitrary domain, then
$$
\mathscr{C}_{\Omega}^{Y,s}\equiv\eta_{\Omega},
$$  
where $s\in Y$.
\end{corollary}
\begin{proof}
Since $Y\subsetneq\mathbb{C}$, there exists
a Hurwitz covering map $g:\mathbb{D}\to Y$.
Now, $\Omega\subsetneq\mathbb{C}$ being a simply connected domain, by Riemann Mapping Theorem, 
we would get a conformal mapping
$h:\,\Omega\to \mathbb{D}$ with $h(w)=0$ and $h'(w)> 0$ for some $w\in\Omega$. Then the composition $g\circ h$
is a holomorphic covering map from $\Omega\setminus\{w\}$
onto $Y\setminus\{s\}$ for some $s\in Y$ that can be extended from $\Omega$ onto $Y$ by taking $w$ to $s$ with its derivative non-zero at the point $w$. The proof now follows by Proposition~\ref{p2prop5}.
\end{proof}


Recall that the hyperbolic density $\lambda_\Omega$, the Hurwitz density $\eta_\Omega$ and the Kobayashi density of the Hurwitz metric $\eta_\Omega^Y$ satisfy the distance decreasing property. Note that, in case of the hyperbolic metric the distance decreasing property is also known as the generalized Schwarz-Pick lemma. Alike to these properties we here show that the Carath\'eodory density of the Hurwitz metric $\mathscr{C}_{\Omega}^{Y,s}$ too satisfies the distance decreasing property.

\begin{theorem}\label{p2thm1}$($Distance decreasing property$)$
Let $\Omega_1,\Omega_2\subset\mathbb{C}$ and $Y\subsetneq\mathbb{C}$ be domains. If there exists a holomorphic function $f$ from $\Omega_1$ into $\Omega_2$ with $f(a)=b,\,f(s)\neq b$ for all $s\in\Omega_1\setminus\{a\}$, then
$$
\mathscr{C}_{\Omega_2}^{Y,c}(f(a))|f'(a)|\leq\mathscr{C}_{\Omega_1}^{Y,c}(a),
$$ 
where $c\in Y$.
\end{theorem}
\begin{proof}
If $\mathcal{H}_b^c(\Omega_2,Y)=\emptyset$, then $\mathscr{C}_{\Omega_2}^{Y,c}=0$ and hence there is nothing to prove. Therefore, without loss of generality we assume that $\mathcal{H}_b^c(\Omega_2,Y)\neq \emptyset$.

By the definition of $\mathscr{C}_{\Omega_2}^{Y,c}(b)$, for every $\epsilon>0$ there exists a holomorphic function $h$ from $\Omega_2$ into $Y$ with $h(b)=c, h(s)\neq c$ for all $s\in\Omega_2\setminus\{b\}$ for some $c\in Y,$ such that 
\begin{equation}\label{p2eq4}
\mathscr{C}_{\Omega_2}^{Y,c}(b)-\epsilon\leq\eta_Y(h(b))|h'(b)|.
\end{equation}
Suppose that $f$ is a holomorphic function from $\Omega_1$ into $\Omega_2$ with $f(a)=b,\,f(s)\neq b$ for all $s\in\Omega_1\setminus\{a\}$.
Now the composition function $h\circ f\in\mathcal{H}(\Omega_1,Y)$ satisfies $(h\circ f)(a)=c$.
Furthermore, $(h\circ f)(t)\neq c$ for all $t\in\Omega_1\setminus\{a\}$ as $b\notin f(\Omega_1)\setminus\{a\}$ and $c\notin h(\Omega_2)\setminus\{b\}$. Now, by plugging the map $h\circ f$ into the definition of $\mathscr{C}_{\Omega_1}^{Y,c}(a),$ it follows that
\begin{equation}\label{p2eq5}
\mathscr{C}_{\Omega_1}^{Y,c}(a)
\geq\eta_{Y}((h\circ f)(a))|(h\circ f)'(a)|
=\eta_{Y}(h(b))|h'(b)||f'(a)|.
\end{equation}
Combining \eqref{p2eq4} and \eqref{p2eq5}, we obtain
$$
\mathscr{C}_{\Omega_1}^{Y,c}(a)\geq(\mathscr{C}_{\Omega_2}^{Y,c}(b) - \epsilon)|f'(a)|
$$
which holds for every $\epsilon>0$.
Letting $\epsilon\to 0$, we have the desired inequality. 
\end{proof}

As a direct consequence of Theorem~\ref{p2thm1}, we obtain the conformal invariance property and monotonicity property of the Carath\'eodory density of the Hurwitz metric $\mathscr{C}_\Omega^{Y,s}$ as follows:

\begin{corollary}$($Conformal invariance property$)$
If $f$ is a conformal mapping from a domain $\Omega_1\subset
\mathbb{C}$ onto another domain $\Omega_2\subset\mathbb{C}$, then for a base domain $Y\subsetneq\mathbb{C}$ we have

$$
\mathscr{C}_{\Omega_2}^{Y,s}(f(w))|f'(w)|=\mathscr{C}_{\Omega_1}^{Y,s}(w),
$$
for all $w\in\Omega_1$ and $s\in Y$.
\end{corollary}

\begin{corollary}$($Domain monotonicity property$)$
If $\Omega_1\subsetneq\Omega_2$ and $Y$ are domains as in Theorem~$\ref{p2thm1}$, then $\mathscr{C}_{\Omega_2}^{Y,s}(w)\leq\mathscr{C}_{\Omega_1}^{Y,s}(w)$ for all $w\in\Omega_1$ and $s\in Y$.
\end{corollary}


Until now, we studied the properties of the Carath\'eodory density of the Hurwitz metric $\mathscr{C}_{\Omega}^{Y,s}$ by fixing the base domain $Y$. For two different base domains, the comparison result is given below.

\begin{theorem}\label{p2thm12}
Let $Y_1,Y_2\subsetneq\mathbb{C}$ and $\Omega\subset\mathbb{C}$ be subdomains. If for every point $b\in Y_2$ there exists a point $a\in Y_1$ 
and a holomorphic covering map $g_b:Y_1\setminus\{a\}\to Y_2\setminus\{b\}$ which extends to the holomorphic function  with $g_b(a)=b$ and $g'_b(a)\neq 0$, then
$$
\mathscr{C}_{\Omega}^{Y_1,a}(w)\leq\mathscr{C}_{\Omega}^{Y_2,b}(w)
$$
for all $w\in\Omega$.
\end{theorem}
\begin{proof}
By the distance decreasing property for Hurwitz density, it follows that
\begin{equation}{\label{p2eq6}}
\eta_{Y_2}(g_b(a))|g'_b(a)|=\eta_{Y_1}(a)
\end{equation}
since $g_b$ is the extended holomorphic covering map from $Y_1$ onto $Y_2$.
Let $\epsilon>0$ be arbitrary. 

If $\mathcal{H}_w^a(\Omega,Y_1)=\emptyset$, then $\mathscr{C}_{\Omega}^{Y_1,a}=0$ and hence there is nothing to prove. Therefore, without loss of generality we assume that $\mathcal{H}_w^a(\Omega,Y_1)\neq \emptyset$.

By the definition of $\mathscr{C}_{\Omega}^{Y_1,a}$, for $a\in Y_1$ and
$w\in\Omega$, there exists a function $h\in\mathcal{H}_w^a(\Omega,Y_1)$ such that
\begin{equation}\label{p2eq7}
\mathscr{C}_{\Omega}^{Y_1,a}(w)\leq\eta_{Y_1}(h(w))|h'(w)|+\epsilon.
\end{equation}
Now we notice that the composed function $g_b\circ h\in\mathcal{H}(\Omega,Y_2)$ satisfies $(g_b\circ h)(w)=b$, $(g_b\circ h)(z)\neq b$ for all $z\in\Omega\setminus\{w\}$. Hence, $g_b\circ h\in\mathcal{H}_w^b(\Omega,Y_2)$. Applying $g_b\circ h$ in the definition of $\mathscr{C}_{\Omega}^{Y_2,b}(w)$, we conclude that
\begin{equation}\label{p2eq8}
\mathscr{C}_{\Omega}^{Y_2,b}(w)\geq\eta_{Y_2}((g_b\circ h)(w))|(g_b\circ h)'(w)|
=\eta_{Y_2}(g_b(a))|g_b'(a)||h'(w)|
\end{equation} 
for all $w\in\Omega$. Combining \eqref{p2eq6}, \eqref{p2eq7}, \eqref{p2eq8} and applying the chain rule, we obtain
$$
\mathscr{C}_{\Omega}^{Y_1,a}(w)\leq\eta_{Y_1}(a)|h'(w)|+\epsilon=\eta_{Y_2}(g_b(a))|g'_b(a)||h'(w)|+\epsilon\leq\mathscr{C}_{\Omega}^{Y_2,b}(w)+\epsilon
$$
for all $w\in\Omega$.
Since $\epsilon$ is arbitrary, we can let it approach to zero to obtain the desired inequality.
\end{proof}

\begin{corollary}
If $Y_1$ and $Y_2$ are conformally equivalent proper subdomains of $\mathbb{C}$ and $\Omega$ is an arbitrary subdomain of $\mathbb{C}$, then
$$
\mathscr{C}_{\Omega}^{Y_1,a}(w)=\mathscr{C}_{\Omega}^{Y_2,b}(w)
$$
holds for every $w\in\Omega$ and for some $a\in Y_1, b\in Y_2$.
\end{corollary}
\begin{proof}
We consider the inverse image of the conformal mapping in Theorem~\ref{p2thm12} to obtain the reverse inequality $
\mathscr{C}_{\Omega}^{Y_1,a}(w)\ge \mathscr{C}_{\Omega}^{Y_2,b}(w)
$.
\end{proof}


\section{A distance function}
\label{p2sec4}

In this section, we consider the usual distance function associated with the Carath\'eodory density of the Hurwitz metric $\mathscr{C}_\Omega^{Y,s}$ for the domains $Y\subsetneq\mathbb{C}$ and $\Omega\subset\mathbb{C}$.

\begin{definition}\label{p2def4.1}
Let $Y\subsetneq\mathbb{C}$ and $\Omega\subset\mathbb{C}$ be domains. For $w_1,w_2\in\Omega$ and $s\in Y$ define
$$
\mathscr{C}_{\Omega}^Y(w_1,w_2)=\inf\int_{\gamma}\mathscr{C}_{\Omega}^{Y,s}(w)|dw|,
$$
where the infimum is taken over all rectifiable paths $\gamma\subset\Omega$ joining $w_1$ to $w_2$. If $\mathscr{C}_{\Omega}^Y$ defines a metric, then we say $(\Omega,\mathscr{C}_{\Omega}^Y)$ a metric space.
\end{definition}


It is easy to see from Definition~\ref{p2def4.1} that $\mathscr{C}_{\Omega}^Y(w_1,w_1)=0$ and $\mathscr{C}_{\Omega}^Y(w_1,w_2)=\mathscr{C}_{\Omega}^Y(w_2,w_1)$ for any $w_1,w_2\in\Omega.$ Further, it can also be verified that $\mathscr{C}_{\Omega}^Y$ satisfies the triangle inequality. Hence, at least we can say that $\mathscr{C}_{\Omega}^Y$ is a pseudo-metric. At present we do not know whether
$\mathscr{C}_{\Omega}^Y$ defines a metric or not. However, we have a partial solution to this whenever $\Omega\subset Y$.
 
\begin{theorem}\label{main-thm5}
If $\Omega\subset Y\subsetneq\mathbb{C}$ are domains, then $(\mathscr{C}_{\Omega}^Y,\Omega)$ becomes a metric space.
\end{theorem}
\begin{proof}
Since $\mathscr{C}_{\Omega}^Y$ is a pseudo-metric on ${\Omega}$, it is enough to show that $\mathscr{C}_{\Omega}^Y(w_1,w_2)>0$ for two distinct points $w_1,w_2\in\Omega$.
Let $\gamma$ be an arbitrary rectifiable curve joining $w_1$ to $w_2$ in $\Omega$.
Since $\Omega\subset Y,$ plugging the inclusion mapping $i\in\mathcal{H}^w_w(\Omega,Y)$ into the definition of $\mathscr{C}_{\Omega}^{Y,w}(w),$ we conclude that 
$$ 
\int_{\gamma}\mathscr{C}_{\Omega}^{Y,w}(w)|dw|
\geq\int_{\gamma}\eta_Y(i(w))|i'(w)||dw|
=\int_{\gamma}\eta_Y(w)|dw|
$$
By the definition of Hurwitz distance (see \cite{AS1}) between two points, it follows that
$$
\int_{\gamma}\mathscr{C}_{\Omega}^{Y,w}(w)|dw|>\eta_Y(w_1,w_2).
$$
Now, taking infimum over $\gamma$, we obtain
$$
\mathscr{C}_{\Omega}^Y(w_1,w_2)\geq\eta_Y(w_1,w_2)>0,
$$
where the last inequality follows from \cite[Theorem~2.3]{AS1}.
Hence $(\Omega,\mathscr{C}_{\Omega}^Y)$ defines a metric space, completing the proof.
\end{proof}

\medskip
\noindent 
{\bf Acknowledgement.}
The authors would like to thank the referee for his/her careful reading of the manuscript and useful remarks. 
The research work of Arstu is 
supported by CSIR-UGC $($Grant No: 21/06/2015(i)EU-V$)$ and
of S. K. Sahoo is partially 
supported by NBHM, DAE $($Grant No: $2/48 (12)/2016/${\rm NBHM (R.P.)/R \& D II}/$13613)$.

\end{document}